\documentclass[a4paper,12pt]{article}
\usepackage{amssymb}
\usepackage{amsthm}
\usepackage{amsmath}
\usepackage{graphicx}
\usepackage{psfrag}
\usepackage{array,hhline}
\newtheorem{lemma}{Lemma}[section]
\newtheorem{theorem}[lemma]{Theorem}
\newtheorem{proposition}[lemma]{Proposition}

\newtheorem{corollary}[lemma]{Corollary}
\theoremstyle{definition}
\newtheorem{definition}[lemma]{Definition}
\newtheorem{remark}[lemma]{Remark}

\numberwithin{equation}{section}
\numberwithin{figure}{section}

\newcommand{\Nset}{\mathcal{N}}

\newcommand{\Sset}{\mathcal{S}}

\newcommand{\Xset}{\mathcal{X}}

\begin{document}

%%%%%%%%%%%%%%%%%%%%%%%%%%%%%%%%%%%%%%%%%%%%%%%%%%%%%%%%%%%%%%%%%%%%%%%%
\title{\Large{On the extension of a class of Hermite multivariate interpolation problems}}

\date{}

\author{Hakop Hakopian, \  Anush Khachatryan\\ \\
 \textit{Yerevan State University, Yerevan, Armenia}\\   \textit{Institute of Mathematics of NAS RA}}
\date{}
\maketitle

\maketitle

%%%%%%%%%%%%%%%%%%%%%%%%%%%%%%%%%%%%%%%%%%%%%%%%%%%%%%%%%%%%%%%%%%%%%%%%
\begin{abstract}
We characterize the sets of solvability for Hermite multivariate interpolation problems when the sum of multiplicities is at most $2n + 2$, with $n$ the degree of the polynomial space. This result extends an earlier theorem (2000) by one of the authors concerning the case $2n+1$. The latter theorem, in turn, can be regarded as a natural extension of a classical Theorem of Severi (1921). 
\end{abstract}



%%%%%%%%%%%%%%%%%%%%%%%%%%%%%%%%%%%%%%%%%%%%%%%%%%%%%%%%%%%%%%%%%%%%%%%%
%%%%%%%%%%%%%%%%%%%%%%%%%%%%%%%%%%%%%%%%%%%%%%%%%%%%%%%%%%%%%%%%%%%%%%%%
\section{Introduction}

In this paper, we characterize the sets of solvability for Hermite multivariate interpolation problems when the sum of multiplicities is at most $2n + 2$, where $n$ denotes the degree of the polynomial space (see Theorem \ref{th:2n+2}). This result extends a previous one (Theorem \ref{th:2n+1}) that addresses the case of $2n+1$. The latter theorem can, in turn, be regarded as a natural extension of a classical Theorem of Severi (Theorem \ref{th:n+1}). 

One should mention that the correctness of Hermite problems was
considered first in algebraic geometry by M. Nagata \cite{N}. In approximation
theory the investigation was initiated by G. G. Lorentz and R. A. Lorentz \cite{L.L.}, \cite{R.L.}.

It should be noted also that a similar study in the case of  Lagrange bivariate interpolation is more advanced (see \cite{HM}).

\subsection{Hermite multivariate interpolation}

Hermite interpolation by multivariate algebraic polynomials is considered.
The interpolation parameters are the values of a function and its partial
derivatives up to some orders at given points.

Denote for $x :=(x_1 ,\ldots, x_k)$ and the multi-index $\alpha:=(\alpha_1,\ldots,\alpha_k)\in \mathbb Z_+^k$ 
$$x^\alpha:=x_1^{\alpha_1},\ldots,x_k^{\alpha_k},\ |\alpha|=\alpha_1+\cdots+\alpha_k,\ \hbox{and}$$ 
$$D^\alpha f:={\frac{\partial^{|\alpha|} f} {\partial x_1^{\alpha_1}\cdots \partial x_k^{\alpha_k}}}.$$
Set for convenience \begin{equation}\label{choose}n^{(k)}:={n+k\choose k}.\end{equation}
The space of polynomials in $k$ variables of total degree
$\le n$ is denoted by 
\begin{equation}\label{dim}\Pi_n^k=\{\sum_{|\alpha|\le n}c_\alpha x^\alpha\},\ \ \dim \Pi_n^k=n^{(k)}.\end{equation}
The coordinates of points and coefficients of polynomials can be taken either real
or complex, thus leading to a real or complex interpolation problem, respectively.

Consider a set $\Xset$ of distinct interpolation points
\begin{equation}\label{set}\Xset:=\Xset_s=\{x^{(1)},\ldots,x^{(s)}\},\ \hbox{where}\ x^{(i)}=(x^{(i)}_1,\ldots,x^{(i)}_k).\ \end{equation}  

Consider also a \emph{scheme}, i.e., a sequence of multiplicities of the points associated with  the degree of interpolating polynomials $n:$ 
$$\Nset:=\Nset^k:=\{n_1,\ldots,n_s; n\}^k.$$
As a rule we assume that the multiplicities in the scheme $\Nset$ are non-increasing:
\begin{equation*}\label{dec}n_1 \ge n_2\ge \cdots\ge n_s.\end{equation*}

In \emph{Hermite interpolation problem}  $(\Nset,\Xset)^k,$ where $k$ is the dimension,
the \emph{interpolation parameters} at a given point $x^{(i)},\ 1\le i\le s,$ are the value of a polynomial and its partial
derivatives up to the order $n_i-1:$ 
\begin{equation}\label{int cond H}
D^{\alpha}p(x^{(i)}) = \lambda_i^\alpha \quad \forall |\alpha|\le n_i-1,\ \  i = 1, \ldots, s.
\end{equation}
The problem of finding a polynomial $p \in \Pi_n^k$ satisfying the above conditions, 
for a data $${\Lambda}:=\{\lambda_i^\alpha \ |\alpha|\le n_i-1,\ 1\le i\le s\},$$ is called  Hermite interpolation problem $(\Nset,\Xset).$

From now on we express briefly the conditions \eqref{int cond H} as follows
$$ D^\Nset p|_\Xset=\Lambda.$$
Points whose multiplicity is equal to one are called \emph{simple}. At these points only the value of a polynomial is interpolated in the conditions \eqref{int cond H}.

It is sometimes convenient, for example in proofs, to allow points of multiplicity $0$ in the problem $(\Nset,\Xset)^k.$ In fact, such a point can be ignored because there is no interpolation condition.

\begin{definition}
A problem  $(\Nset,\Xset)^k$ is called
$n$-\emph{solvable} if for any data $\Lambda$ there exists a
polynomial (not necessarily unique) $p \in \Pi_n^k$, satisfying 
\eqref{int cond H}.
\end{definition}
A polynomial $p \in \Pi_n^k$ is called an $n$-\emph{fundamental polynomial}
for a parameter $D^{\beta}p(x^{(j)}),\ |\beta|\le n_{j}-1,\  1\le j\le s,$ if it is a solution to the problem 
\eqref{int cond H} where $\lambda_i^\alpha=0 \ \ \forall (i,\alpha) \neq (j,\beta)\ \hbox{and}\ \lambda_{j}^{\beta}=1.$

The above fundamental polynomial we denote by 
$$p_{j}^{\star \beta}:=p_{j\Nset\Xset}^{\star \beta}.$$
Let us mention that sometimes we call $n$-fundamental a polynomial, which is a solution to the problem \eqref{int cond H} where $\lambda_i^\alpha=0 \ \ \forall (i,\alpha) \neq (j,\beta)\ \hbox{and}\ \lambda_{j}^{\beta}\neq 0,$
since it equals a non-zero constant times $p_{j}^{\star \beta}.$
\begin{definition}
A problem
$(\Nset,\Xset)^k$ is called
$n$-\emph{independent} if for each parameter in \eqref{int cond H} there exists an $n$-fundamental polynomial.
\end{definition}
The Lagrange formula for the above problem $(\Nset,\Xset)^k,$ giving a solution to the problem \eqref{int cond H}, takes the form:
$$p=\sum_{i=1}^s \sum_{|\alpha|\le n_i-1} \lambda_i^\alpha p_{i}^{\star \alpha}.$$
From here we get:
\begin{proposition} \label{indep} A problem  $(\Nset,\Xset)^k$ is $n$-independent 
if and only if it is $n$-solvable.
\end{proposition}

We have that  the number of interpolation conditions at a point $x^{(i)}$ in \eqref{int cond H} equals to ${(n_i-1)}^{(k)}$ (see \ref{choose}). Hence the total number of the conditions equals to 
\begin{equation}\label{|Nset|}\Nset_\#^k:=\sum_{i=1}^s{(n_i-1)}^{(k)}.\end{equation}
It is easily seen that fundamental polynomials are linearly independent.  Therefore a
necessary condition of $n$-independence, or $n$-solvability, of the problem  $(\Nset,\Xset)^k$ is   
$$\Nset_\#^k \le n^{(k)}.$$
A scheme $\Nset:=\Nset^k$ satisfying the above condition we call 
 \emph{$(\le)$-scheme}.

A scheme $\Nset:=\Nset^k$ satisfying the condition 
$$\Nset_\# ^k< n^{(k)}$$
we call  \emph{$(<)$-scheme}.
\begin{definition}
A problem  $(\Nset,\Xset)^k$ is called
$n$-\emph{correct} if for any data $\Lambda$ there exists a
\emph{unique} polynomial $p \in \Pi_n^k$, satisfying the conditions
\eqref{int cond H}.
\end{definition}
A necessary condition of
$n$-correctness is: 
\begin{equation}\label{exact}\Nset_\#^k = n^{(k)}.\end{equation}
A scheme $\Nset:=\Nset^k$ satisfying the above condition we call  \emph{exact scheme.}
\begin{proposition} \label{correct H}
Assume that a scheme $\Nset^k$ is exact. Then the $n$-correctness of the problem $(\Nset,\Xset)^k$ is equivalent to each of the following conditions:
\begin{enumerate}
\item
The problem $(\Nset,\Xset)^k$ is $n$-solvable.
\item
$p \in \Pi_n^k,\ D^\Nset p|_\Xset=0\implies p = 0.$
\end{enumerate}
\end{proposition}
\begin{proof} Indeed, each of the $\Nset_\#^k$ conditions  in \eqref{int cond H} can be considered as a linear equation with respect to $n^{(k)}$ unknown coefficients of $p\in\Pi_n^k.$ Thus for exact scheme the conditions \eqref{int cond H} reduce to a square linear system (see \eqref{exact}). Now, the condition (i) means that the linear system is consistent  for any right-hand side values, while the condition (ii) means that the homogeneous linear system has a single trivial solution. It is well known that in the case of a square system each of these properties implies that the system has a unique solution for any right-hand side values. The later property means that the problem $(\Nset,\Xset)^k$ is $n$-correct. \end{proof}

\begin{proposition} \label{prop:(<)}
Assume  that in the problem $(\Nset,\Xset)^k,$ the scheme $\Nset^k$ is 
$(<)$-scheme. Then 
there exists $$p \in \Pi_n^k,\ p\neq 0,\ \hbox{and}\ D^\Nset p|_\Xset=0.$$
\end{proposition}
\begin{proof} Indeed, in this case the conditions \eqref{int cond H} reduce to a linear homogeneous system, where the number of equations is less than the number of unknowns. As we know such systems have non-trivial solutions.
\end{proof}
\noindent Thus, for the $(<)$-scheme $\Nset=\{1,1,1,1,1;2\}^2$ we obtain   
\begin{corollary}\label{lem:11111} For any five points in the plane, there exists a conic that passes through all of them. \end{corollary}

For a $(<)$-scheme $\Nset:=\{n_1,\ldots,n_s; n\}^k$ we define the associated exact
scheme
\begin{equation}\label{wideN}\widehat\Nset:=\{n_1,\ldots,n_s,\underbrace{1,\ldots,1}_\text{$r$}; n\}^k,\ \hbox{where}\  r =n^{(k)}-\Nset_\#^k.\end{equation} 
Accordingly, for a set
$\Xset = \{x^{(1)},\ldots, x^{(s)}\},$ we define 
\begin{equation}\label{XcupS}\widehat\Xset: =\Xset\cup \Sset= \{x^{(1)},\ldots,  x^{(s+r)}\},\end{equation}
where $\Sset=\{x^{(s+1)},\ldots, x^{(s+r)}\}$ is a set of $r$ simple points.

\begin{proposition}[\cite{Mys}, Lem.2.1]\label{widehat} The problem $(\widehat\Nset,\widehat\Xset)^k$ is $n$-correct, for some $\widehat\Xset,$ if and only if the
problem $(\Nset,\Xset)^k$ is $n$-solvable.
\end{proposition}
For the sake of completeness we present: 
\begin{proof}
The direction ``only if" is evident. For the direction ``if", in view of Proposition \ref{correct H}(i), it is enough to prove that there exists a (simple) point $x^{(s+1)}$ such that
the problem $(\Nset',\Xset')^k,$ where $\Xset'=\Xset\cup \{x^{(s+1})\}$
and $\Nset:=\{n_1,\ldots,n_s,1; n\},$ is $n$-solvable. For this, in view of the $n$-solvability of $(\Nset,\Xset)^k,$  we need just to have an $n$-fundamental polynomial of the simple point $x^{(s+1)}.$

Now, according to Proposition \ref{prop:(<)}, there exists
$$p \in \Pi_n^k,\ p\neq 0,\ \hbox{and}\ D^\Nset p|_\Xset=0.$$
Since $p\neq 0$ there is a point: $x^{(s+1)}$ such that $p(x^{(s+1)})\neq 0.$

Then, it is easily seen that the polynomial $p$ is the desired $n$-fundamental polynomial
of $x^{(s+1)}.$
\end{proof}
Denote  $$\Pi_{\Nset\Xset}^k:=\left\{p\in\Pi_n^k\ |\ D^\Nset p|_\Xset=0\right\}.$$
\begin{corollary}\label{dimNX} Suppose that  the
problem $(\Nset,\Xset)^k$ is $n$-solvable. Then we have that
$$\dim\Pi_{\Nset\Xset}^k=n^{(k)}-\Nset_\#^k.$$
 \end{corollary}
\begin{proof}
Assume, in view of Proposition \ref{widehat}, that the problem  
$(\widehat\Nset,\widehat\Xset)^k$ is $n$-correct. Then we readily get that
$$\Pi_{\Nset\Xset}^k=\hbox{\emph{Linear span}}\{p_{i\widehat\Nset\widehat\Xset}^{\star 0}\}_{i=s+1}^{s+r}.$$ 
In view of \eqref{wideN} it remains to note that the fundamental polynomials are linearly independent.
\end{proof}

\begin{remark}[\cite{Hak2}, rem.1.14] \label{rem:XX'}Let $\Xset_s = \{x^{(i)}\}_{i=1}^s$ and $\tilde\Xset_s = \{\tilde x^{(i)}\}_{i=1}^s.$
\begin{enumerate}
\item
If the problem $(\Nset,\Xset)^k$ is $n$-solvable, then so is the problem $(\Nset,\tilde\Xset)^k$  with $\tilde x^{(i)}$ sufficiently close to $x^{(i)},\  i = 1, \ldots, s.$
\item
If the problem $(\Nset,\Xset)^k$ is $n$-solvable, with $x^{(i)}$ all belonging to an $m$-dimensional
linear subspace $L,\ 1\le m<k,$ then so is the problem $(\Nset,\tilde X)^k,$ where the
projection of $\tilde x^{(i)}$ on $L$ is $x^{(i)},\ i = 1,\ldots,s.$
\end{enumerate}
\end{remark}
Until the end of this subsection, we study the bivariate case.

In this case we use the following simplified notation:
$$\Nset:=\Nset^2,\ \ \Nset_\#:=\Nset_\#^2,\ \ (\Nset,\Xset):=(\Nset,\Xset)^2,\ \ \Pi_n:=\Pi_n^2,\ \  \Pi_{\Nset\Xset}:=\Pi_{\Nset\Xset}^2.$$
In this case instead of \eqref{choose} it is more convenient to use the notation
$$\overline n:={(n-1)}^{(2)}=1+\cdots+n.$$
Then the equality \eqref{|Nset|} and the second equality of \eqref{dim} become: 
\begin{equation*}\label{|Nset|2}\Nset_\#:=\sum_{i=1}^s\overline{n_i},\ \ \dim\Pi_n=\overline{n+1}.\end{equation*}

A \emph{plane algebraic curve} is the zero set of some bivariate polynomial of degree $\ge 1.$~To simplify notation, we shall use the same letter,  say $p$,
to denote the polynomial $p$ and the curve given by the equation $p(x,y)=0$.
In particular, by $\alpha$ we denote a linear 
polynomial from $\Pi_1$ and the line defined by the equation
$\alpha(x, y)=0.$ Similarly, by $\beta$ we denote a quadratic 
polynomial from $\Pi_2$ and the conic defined by the equation
$\beta(x, y)=0.$ 

\begin{definition}\label{qNX}Let $(\Nset,\Xset)$ be a problem, where $\Xset$ is a set given by \eqref{set}, and $f$ be a plane curve of degree $m.$ Then we define the scheme $\Delta^f_\Xset$ as follows 
\begin{equation}\label{delta}\Delta^f_\Xset:=\{\delta(1),\ldots,\delta(s);m\},\end{equation}
where $\delta(i)=1$ if $x^{(i)}\in f,$ and $\delta(i)=0$ otherwise.
\end{definition}

Denote by $f_{\Nset,\Xset}^\#$ the number of points from $\Xset$ on the curve $f,$ counted with the multiplicities from $\Nset:$
\begin{equation}\label{N-N+1} f_{\Nset,\Xset}^\#=\sum_{i=1}^s \delta(i)n_i= \sum_{i=1}^s\overline n_i-\sum_{i=1}^s\overline {n_i-\delta(i)}.\end{equation}
Note that this notation, as well as the previous definition, mainly will be used in the cases where the curve $f$ either is a line or a conic.

For $\Nset=\{n_1,\ldots,n_s; n\}$ denote $\Nset-1:=\{n_1-1,\ldots,n_s-1; n-1\}.$

\noindent We will need the following well-known proposition:
\begin{proposition} \label{prop:line}  
Let $\Nset=\{n_1,\ldots,n_s; n\}$ be a scheme with 
$n_1+\cdots+n_s= n+1$ and $\Xset$ be a set of points with $\Xset\subset \alpha^0,$ where $\alpha^0$ is a line. Then 
$$p\in\Pi_n,\ D^\Nset p|_\Xset=0\implies p=\alpha^0 q,$$
where $q\in\Pi_{n-1}$ and $D^{\Nset-1} q|_\Xset=0.$
\end{proposition}

\begin{corollary} \label{cor:line+}  
Suppose $(\Nset,\Xset)$ is a problem, where $\Nset=\{n_1,\ldots,n_s; n\}$ is $(\le)$-scheme. Suppose also $\alpha$ is a line with $\alpha^\#_{\Nset,\Xset}=n+1.$ Then 
the problem $(\Nset,\Xset)$ is $n$-solvable if and only if the problem $(\Nset^0,\Xset)$ is $(n-1)$-solvable, where $\Nset^0:=\Nset-\Delta^\alpha_\Xset=\{n_1-\delta(1),\ldots,n_s-\delta(s); n-1\}.$
\end{corollary}

\subsection{Some results in the multivariate case}
Set $$d(n, k) := \overline{n+1}-\overline{n-k+1}=(n-k+2)+(n-k+3)+\cdots+(n+1).$$
The following proposition will be used for lines and conics. 
\begin{proposition}\label{fn1n2} Let $\Nset:=\{n_1,\ldots,n_s; n\}$ be a scheme. If the problem $(\Nset,\Xset)^k,\ k\ge 2,$ is $n$-solvable for a set of points $\Xset,$ then for any plane curve $f$ of degree $m\le n$ we have that
$f_{\Nset,\Xset}^\# \le d(n,m).$
\end{proposition}
\begin{proof} First consider the case $k=2.$ Let the problem $(\Nset,\Xset)$ be $n$-solvable for a set of points $\Xset.$ 
Assume, by way of contradiction, that for a curve $f^0$ of degree $m\le n,$ we have that
 $f_{\Nset,\Xset}^{0\#} >d(n,m).$

Then,
Proposition \ref{widehat} implies that
\begin{equation}\label{WN0X0}(\widehat\Nset,\widehat\Xset)\ \hbox{is $n$-correct,}\end{equation} 
for some $\widehat\Xset=\Xset\cup\Sset,$ 
where $\Sset$ is a set of $r$ simple points with $r=\overline{n+1}- \Nset_\#^2$ and $$\widehat\Nset:=\{n_1,\ldots,n_s,\underbrace{1,\ldots,1}_\text{$r$}; n\}.$$ 
Moreover, by making use of
Remark \ref{rem:XX'} we may assume that
the set $\Sset$ contains no point from the curve $f^0.$

Now, consider the following scheme 
$$\widehat\Nset^0:=\widehat\Nset - \Delta^{f^0}_{\widehat\Xset}=\{n_1^0,\ldots,n_s^0,\underbrace{1,\ldots,1}_\text{$r$}; n-m\}$$ $$:=
\{n_1-\delta(1),\ldots,n_s-\delta(s),\underbrace{1,\ldots,1}_\text{$r$}; n-m\}.$$
The above condition on the curve $f^0,$ in view of \eqref{N-N+1}, means that
\begin{equation}\label{f02n+1}  \sum_{i=1}^s\delta(i) n_i>d(n,m).\end{equation} 
Then let us check that $\widehat\Nset^0$ is $(<)$-scheme. Indeed, by using \eqref{N-N+1} and \eqref{f02n+1}, we obtain  
\begin{equation*}\widehat\Nset^0_\# =\sum_{i=1}^s\overline {n_i-\delta(i)}+r< \sum_{i=1}^s\overline n_i+r -d(n,m) = \overline{n+1}-d(n,m)=\overline {n-m+1}.\end{equation*}
Therefore, in view of Proposition \ref{prop:(<)}, there exists a polynomial $p^0\in\Pi_{n-m}$ such that
$$D^{\widehat \Nset^0}p^0|_{\widehat\Xset}=0.$$

\noindent Now we readily get that the polynomial $p:=f^0p^0\in\Pi_n$ satisfies the condition: $$D^{\widehat \Nset}p|_{\widehat\Xset}=0.$$
This, by Proposition \ref{correct H}(ii), contradicts \eqref{WN0X0}. 

Now consider the case of general dimension: $(\Nset,\Xset)^k.$ Assume that $\Xset$ does not satisfy the condition, i.e., there is a plane curve $f^0$ passing
through more than $d(n,m)$ of its points. Consider the respective subproblem restricted to
the interpolation points belonging to the curve $f^0:\ (\Nset_0,\Xset_0)^k.$ As we proved above, this subproblem is not $n$-solvable in the case of $k=2.$ Therefore the original problem also is not $n$-solvable.
\end{proof}
By applying Proposition \ref{fn1n2} for the cases $k=1,2,$ we obtain the following two necessary conditions for $n$-solvability:
\begin{corollary}\label{nes2} Let $\Nset:=\{n_1,\ldots,n_s; n\}$ be a scheme. If the problem $(\Nset,\Xset)^k,\ k\ge 2,$ is $n$-solvable for a set of points $\Xset,$ then: 
\begin{enumerate}
\item
For any line $\alpha$ we have that
$\alpha_{\Nset,\Xset}^\# \le n + 1.$
\item
For any conic $\beta$ we have that
$\beta_{\Nset,\Xset}^\# \le 2n + 1.$
\end{enumerate}
\end{corollary}
In the sequel the condition (i) will be referred to as the $l/(n + 1)$-\emph{condition}.

\noindent Similarly,  the condition (ii) will be referred to as the $c/(2n + 1)$-\emph{condition}.
\begin{definition}\label{Nindep}
A scheme $\Nset:=\Nset^2$ is called \emph{regular} if there exists a set of points $\Xset$ such that the problem  $(\Nset,\Xset)^2$ is 
$n$-independent.
\end{definition}

By taking into account the above condition (i) for the lines passing through any two points, and the condition (ii) for the conics passing through any five points, one readily obtains 
\begin{corollary}\label{cor:n1n2} Let $\Nset:=\{n_1,\ldots,n_s; n\}^2$ be a scheme with non-increasing sequence of
multiplicities. If $\Nset$ is regular then we have that
$$n_1 + n_2 \le n + 1\ \ \hbox{and}\ \ n_1+\cdots+n_5\le 2n+1.$$
\end{corollary}
Now we are in a position to present the result, which we are going to extend. Actually, it is slightly reformulated (see \cite{Hak}, 2nd par. after the proof of Thm.9).
\begin{theorem}[\cite{Hak}, Thm.9] \label{th:2n+1} Let $\Nset:=\{n_1,\ldots,n_s; n\}^k,\ k\ge 2,$ be a scheme, satisfying the condition
\begin{equation*}\label{2cond} \sum_{i=1}^sn_i\le 2n+1.\end{equation*}
Then the problem $(\Nset,\Xset)^k$ is
$n$-solvable if and only if there is no line passing through more than $n+1$ points of $\Xset.$
\end{theorem}
This theorem can be regarded as a natural extension of the following classical Theorem of Severi (1921).
\begin{theorem}[Severi \cite{S}, Thm.9] \label{th:n+1} Let $\Nset:=\{n_1,\ldots,n_s; n\}^k,\ k\ge 1,$ be a scheme. Then the problem $(\Nset,\Xset)^k$ is $n$-solvable for any point set $\Xset$ if and only if  
\begin{equation*}\label{1cond}  \sum_{i=1}^sn_i\le n+1.\end{equation*}
\end{theorem}
Until the end of the subsection, we study the bivariate case again.

From Theorem \ref{th:2n+1} we get
\begin{corollary} \label{cor:2n+1} Suppose that a scheme $\Nset=\{n_1,\ldots,n_s; n\}^2,$ with non-increasing sequence of multiplicities, satisfies the condition: 
$\sum_{i=1}^sn_i\le 2n+1.$ 
Then $\Nset$ is regular if only if  $n_1 + n_2 \le n + 1.$  
\end{corollary}
\begin{proof} Indeed, the direction ``only if" follows from Corollary \ref{cor:n1n2}. For the direction ``if" consider a problem $(\Nset,\Xset),$ where in $\Xset$ no three points are collinear. Then according to Theorem  \ref{th:2n+1} the problem $(\Nset,\Xset)$ is $n$-solvable. \end{proof}

For $\Nset=\{n_1,\ldots,n_s; n\}$ denote $\Nset-2:=\{n_1-1,\ldots,n_s-1; n-2\}.$

Next proposition is an extension of Proposition \ref{prop:line}:
\begin{proposition} \label{prop:conic} Let $\Nset:=\{n_1,\ldots,n_s; n\}$ be a scheme with non-increasing sequence of
multiplicities such that 
$$n_1+n_2\le n+1, \quad n_1+\cdots+n_s= 2n+1,\ s\ge 5.$$
Suppose $\Xset$ is a set of $s$ points with $\Xset\subset \beta^0,$ where $\beta^0$ is an irreducible conic. Then we have that
\begin{equation}\label{1357}p\in\Pi_n,\ D^\Nset p|_\Xset=0\implies p=\beta^0 q,\end{equation} 
where $q\in\Pi_{n-2}$ and $D^{\Nset-2} q|_\Xset=0.$
\end{proposition}
\begin{proof}
Consider the following scheme: 
$$\Nset^0:=\Nset-2:=\{n_1^0,\ldots,n_s^0; n-2\}:=
\{n_1-1,\ldots, n_s-1; n-2\}.$$
We have that 
$$\sum_{i=1}^sn^0_i=\sum_{i=1}^sn_i-s\le 2n+1-5=2(n-2).$$

Note that 
$$n_1^0+n_2^0:=n_1+n_2-2\le (n-2)+1.$$
The above inequality, in view of irreducibility of $\beta_0,$  implies that the $l/(n-1)$-condition is satisfied for the problem $(\Nset^0,\Xset)$.

Thus, in view of Theorem \eqref{th:2n+1}, we obtain that $(\Nset^0,\Xset)$ is $(n-2)$-solvable. Similarly we obtain that the problem $(\Nset,\Xset)$ is $n$-solvable.

Now, consider the following linear spaces:
$$\Pi_{\Nset\Xset}=\{p\in\Pi_n\ |\ D^\Nset p|_\Xset=0\},$$
$$\beta^0\cdot\Pi_{\Nset^0\Xset}=\{\beta^0q\ |\ q\in\Pi_{n-2},\ D^{\Nset^0} q|_\Xset=0\}.$$
Note that $\beta^0\cdot\Pi_{\Nset^0\Xset}\subset \Pi_{\Nset\Xset}.$
Since the problems $(\Nset, \Xset)$ and $(\Nset^0, \Xset)$ are solvable, we get from Corollary \ref{dimNX}, that
$$\dim \Pi_{\Nset\Xset}=\overline {n+1}- \sum_{i=1}^s\overline n_i=
\overline {n-1}+2n+1- (\sum_{i=1}^s\overline {n_i-1}+2n+1)$$
$$=\overline {n-1}- \sum_{i=1}^s\overline {n_i-1}=\dim\Pi_{\Nset^0\Xset}=\dim (\beta^0\cdot\Pi_{\Nset^0\Xset}).$$
Thus we obtain that $\Pi_{\Nset\Xset}=\beta^0\cdot\Pi_{\Nset^0\Xset}.$
\end{proof}
Below we consider the case $s=4$ of above proposition. Note that the cases $s\le 3$ are not possible.
\begin{remark} \label{rem:conic} Let $\Nset:=\{n_1,\ldots,n_4; n\}$ be a scheme with non-increasing sequence of
multiplicities such that 
$$n_1+n_2\le n+1, \quad n_1+\cdots+n_4= 2n+1.$$
Then $n=2m-1,\ m\ge 1,$ implies that  $\Nset_{2m-1}:=\{m,m,m,m-1; 2m-1\}.$
\noindent While $n=2m$ implies that  $\Nset_{2m}:=\{m+1,m,m,m; 2m\}.$

Now, let  $\Nset=\Nset_s,\ \forall s .$ Note that $\Nset$ is exact. The proof of Proposition \ref{prop:conic} implies that the problem $(\Nset,\Xset)$ is solvable, provided that $\Xset\subset \beta^0,$ with $\beta^0$ any irreducible conic. The latter condition in this case means that no three points of $\Xset$ are collinear. Thus the problem $(\Nset,\Xset)$ is correct:
$$p\in\Pi_n,\ D^{\Nset_s} p|_\Xset=0\implies p=0.$$
Note that also $\Nset^0$ (see the above proof) is correct, since if $\Nset=\Nset_s$ then $\Nset^0=\Nset_{s-2}.$
Therefore Proposition \ref{prop:conic} is valid also in the case $s=4,$ and the statement \eqref{1357}, i.e., $p=\beta^0q,$ reduces to $0=\beta^0\cdot 0.$ 
\end{remark}

\begin{remark} \label{rem:2n+2}
Note that, according to Corollary \ref{cor:2n+1}, for any scheme $\Nset=\{n_1,\ldots,n_s; n\}^2,$ with non-increasing sequence of multiplicities, we have that
the conditions $n_1 + n_2 \le n + 1,\ \ \sum_{i=1}^sn_i\le 2n+1,$ imply that $\Nset^2$ is $(\le)$-scheme. 

In the main result of this paper below we consider instead the condition 
$\sum_{i=1}^sn_i\le 2n+2$ which, in combination with $n_1 + n_2 \le n + 1,$ does not imply that $\Nset^2$ is $(\le)$-scheme. Indeed, for a counterexample, we just add a fifth multiplicity $1$ to the scheme $\Nset_s$  of the above remark. 

This is why, in the following theorem we have an additional condition, compared to Theorem \ref{th:2n+1}: $\Nset^2$ is $(\le)$-scheme.\end{remark}

\section{The main result}

In this section we will prove the following 
\begin{theorem}\label{th:2n+2} Let $\Nset:=\{n_1,\ldots,n_s; n\}^2$ be a 
$(\le)$-scheme satisfying the condition
\begin{equation*}\label{2cond'}\sum_{i=1}^sn_i\le 2n+2.\end{equation*}
Then the problem $(\Nset,\Xset)^k,\ k\ge 2,$ is
$n$-solvable if and only if: 
\begin{enumerate}
\item
There is no line passing through more than $n+1$ points
of $(\Nset,\Xset).$ 
\item
There is no conic passing through more than $2n+1$ points of $(\Nset,\Xset).$ 
\end{enumerate}
\end{theorem}

\begin{proof} The direction ``only if" follows from Corollary \ref{nes2}. Let us prove the direction ``if". First we consider the (main) case $k = 2.$

Note that the above condition (ii) means that
there is no conic passing through all the points of $\Xset.$

In view of Theorem \ref{th:2n+1} we can assume, in what
follows, that 
\begin{equation}\label{hav} \sum_{i=1}^sn_i= 2n+2.\end{equation}

 From now on suppose also that $\Nset:=\{n_1,\ldots,n_s; n\}$ is a scheme with non-increasing sequence of multiplicities:
 \begin{equation}\label{dec'}n_1 \ge n_2\ge \cdots\ge n_s.\end{equation}
By using the condition (i) for the line passing through $x^{(1)},x^{(2)}\in \Xset,$  we get
\begin{equation}\label{fcond}n_1 + n_2 \le n + 1.\end{equation}
We can assume also that
\begin{equation}\label{n1} n_1\le n-1. 
\end{equation}
Indeed, assume that $n_1=n+1.$ Then the only $(\le)$-scheme is the  the Taylor scheme  $\{n+1;n\}.$ But the latter does not satisfy \eqref{hav}.
While if $n_1=n$ then the conditions \eqref{hav} and \eqref{fcond}   imply that 
 $\Nset=\{n,\underbrace{1,\ldots,1}_\text{$n+2$};n\},$ 
which is not  $(\le)$-scheme.

We will apply induction
on the degree of the scheme $\Nset.$ 

{\bf Step 1.}
It is easily seen that  if $n = 1,$ then there is no $(\le)$-scheme satisfying the condition \eqref{hav}.

In the case $n = 2,$ there is only one $(\le)$-scheme 
$\Nset=\{1, 1, 1, 1, 1, 1; 2\}$ satisfying the conditions \eqref{hav} and \eqref{n1}. According to the condition (ii) of Theorem, the six simple points of the set $\Xset$ do not belong to a conic. Since $\Nset$ is exact scheme, we obtain from Proposition \ref{correct H}(ii), that $\Xset$ is $2$-correct. 

Now, assuming the statement holds for all degrees $\le n-1,$ we will prove it for $n,$ where $n\ge 3.$

{\bf Step 2.} Assume that there is a line $\alpha^0$ which passes through $n+1$ points of $(\Nset,\Xset),$ i.e., the following condition is satisfied:
\begin{equation}\label{elln+1}\alpha^{0\#}_{\Nset,\Xset}=n+1.
\end{equation} 
Consider the scheme (see \eqref{delta})
$$\Nset^0:=\Nset - \Delta^{\alpha^0}_\Xset=\{n_1^0,\ldots,n_s^0; n-1\}:=
\{n_1-\delta(1),\ldots,n_s-\delta(s); n-1\}.$$
From \eqref{N-N+1}, in view of the condition \eqref{elln+1}, we have that 
\begin{equation}\label{alphan+1}  \sum_{k=1}^s\overline n_k=\sum_{k=1}^s\overline {n_k-\delta(k)}+n+1.\end{equation} 
Next, we want to apply the induction hypothesis for the problem $(\Nset^0,\Xset).$
For this let us check first that $\Nset^0$ is $(\le)$-scheme. Indeed, by using \eqref{alphan+1}, we get  
\begin{equation*}\sum_{k=1}^s\overline n_k^0= \sum_{k=1}^s\overline {n_k-\delta(k)}= \sum_{k=1}^s\overline n_k -(n+1) \le \overline{n+1}-(n+1)=\overline n.\end{equation*}
The condition \eqref{2cond'} for $\Nset^0,$ readily follows from the fact that $\alpha^0$ passes through at least two points of $\Xset:$
\begin{equation}\label{-2}\sum_{i=1}^s n_i^0= \sum_{i=1}^s\ n_i -\sum_{i=1}^s\delta(i) \le 2n+2-2=2n.\end{equation}
Then let us verify that the $l/n$-condition is satisfied with respect to the problem $(\Nset^0,\Xset),$ i.e., there is no 
line $\alpha$ containing $\ge(n + 1)$ points of $(\Nset^0,\Xset).$ Assume conversely that there is such a
line $\alpha.$ Then evidently $\alpha\neq\alpha^0$ and, in view of the relation
$$\alpha_{\Nset^0,\Xset}^\#\le \alpha_{\Nset,\Xset}^\#,$$
we obtain that $\alpha$ contains $n+1$ points from $(\Nset,\Xset).$ 
Since $\alpha$ contains $n+1$ points both from $(\Nset,\Xset)$  and $(\Nset^0,\Xset),$ we conclude that the line $\alpha$ does not intersect $\alpha_0$ at a point of $\Xset.$
Therefore, the (reducible) conic $\beta=\alpha^0\alpha$ passes through $2n+2$ points of $\Xset,$ which contradicts the condition (ii) of Theorem.

Finally let us verify the $c/(2n-1)$-condition for the problem $(\Nset^0,\Xset).$ Assume, by way of contradiction, that  there is a 
conic $\beta$ containing $\ge 2n$ points of $(\Nset^0,\Xset).$
In view of \eqref{-2} this means that $\beta$ passes through all the points $x^{(i)}\in\Xset$ with $n^0_i\ge 1.$

Note that if $\alpha^0$ passes through at least three points of $\Xset$ then
in \eqref{-2} we would obtain $\sum_{i=1}^s n_i^0\le 2n-1$ and hence the
$c/(2n-1)$-condition is satisfied. It remains to consider the case where 
$\alpha^0$ passes through exactly two points of $\Xset$, say, $x^{(q)}$ and $x^{(r)}.$  In this case, in view of \eqref{elln+1}, we have that $n_q+n_r=n+1.$ According to \eqref{n1} we get that  $n_q, n_r\ge 2.$ Hence we have that $n_i^0\ge 1 \ \forall 1\le i\le s.$ Thus $\beta$ passes through all the points of $\Xset,$ which is contradiction.

Therefore, the problem $(\Nset^0,\Xset)$ is $(n-1)$-solvable by virtue of the induction hypothesis. 

Now, Proposition \ref{widehat} implies that there is a set $\Sset_0,$ consisting of $r:=\overline n- \Nset^0_\#$ simple points (see \eqref{wideN} and \eqref{XcupS}), such that the problem
\begin{equation}\label{N0X0}(\widehat\Nset^0,\widehat\Xset)\ \hbox{is $(n-1)$-correct,}\end{equation} 
where  
$$\widehat \Nset^0=\{n_1^0,\ldots,n_s^0,\underbrace{1,\ldots,1}_\text{$r$};n-1\},\quad \widehat\Xset=\Xset\cup\Sset_0.$$

Here, by making use of
Remark \ref{rem:XX'}, we can also assume that
the set $\Sset_0$ contains no point belonging to the line $\alpha_0.$

Now consider the problem $(\widehat\Nset,\widehat\Xset),$
where
$\widehat \Nset=\{n_1,\ldots,n_s,\underbrace{1,\ldots,1}_\text{$r$};n\}.$

\noindent Let us check first that $\widehat\Nset$ is exact scheme. Indeed, in view of \eqref{alphan+1}, we have 
\begin{equation*}\label{abc}\widehat\Nset_\#=\Nset_\#+\#\Sset_0=\Nset_\#^0+(n+1)+\#\Sset_0=\overline n+(n+1)=\overline{n+1}.\end{equation*}
\noindent Next, let us prove that the problem $(\widehat\Nset,\widehat\Xset)$ is $n$-correct. For this, let us assume, in view of Proposition \ref{correct H}(ii), that
$D^{\widehat \Nset}p|_{\widehat\Xset}=0,$ where $p\in\Pi_n$
and prove that $p=0.$

\noindent By the choice of the set $\Sset_0$ and the condition \eqref{elln+1} we have that
$\alpha^{0\#}_{\widehat\Nset,\widehat\Xset}=n+1.$

\noindent Then, Proposition \ref{prop:line} implies that $p=\alpha_0 q,$ where
$$q\in\Pi_{n-1}\ \hbox{and}\  D^{\widehat \Nset^0}q|_{\widehat\Xset}=0.$$
Now, in view of \eqref{N0X0} and Proposition \ref{correct H}(ii), we get $q = 0,$ and hence $p = 0.$

Thus, the problem $(\widehat\Nset,\widehat\Xset)$ is $n$-correct, which implies evidently that $(\Nset,\Xset)$ is $n$-solvable.  \hfill \qed
 
From now on, in view of Step 2, we can assume that the $l/n$-condition is satisfied for the scheme $\Nset,$ which implies that
\begin{equation}\label{elln12}n_1+n_2\le n.\end{equation}

{\bf Step 3.} 
Assume that there is a line $\alpha^0$ passing through at least  $3$ points of $\Xset.$
Consider the following scheme (see \eqref{delta}) 
$$\Nset^0:=\Nset - \Delta^{\alpha^0}_\Xset=\{n_1^0,\ldots,n_s^0; n-1\}:=
\{n_1-\delta(1),\ldots,n_s-\delta(s); n-1\}.$$
We have that
$$\sum_{i=1}^s n_i^0= \sum_{i=1}^s\ n_i -\sum_{i=1}^s\delta(i) \le 2n+2-3= 2(n-1)+1,$$
On the other hand, in view of Step 2, the $l/n$-condition holds.
Thus according to Theorem \ref{th:2n+1} we obtain that the problem $(\Nset^0,\Xset)$ is $(n-1)$-solvable. 

Now, Proposition \ref{widehat} implies that there is a set $\Sset_0,$ consisting of $r_0:=\overline n- \Nset^0_\#$ simple points, such that the problem
\begin{equation}\label{N0X03}(\widehat\Nset^0,\widehat\Xset)\ \hbox{is $(n-1)$-correct,}\end{equation} 
where  
$$\widehat \Nset^0=\{n_1^0,\ldots,n_s^0,\underbrace{1,\ldots,1}_\text{$r_0$};n-1\},\quad \widehat\Xset=\Xset\cup\Sset_0.$$  
Again, by making use of
Remark \ref{rem:XX'}, we can assume also that
the set $\Sset_0$ contains no point belonging to the line $\alpha_0.$

Then denote by $\Sset_1$ a set of $r_1:=n + 1 - \alpha^{0\#}_{\Nset,\Xset}$
simple points in $\alpha^0\setminus \Xset.$ 

Consider
the problem $(\widetilde\Nset,\widetilde\Xset),$  where 
$\widetilde\Xset:=\Xset\cup \Sset_0\cup \Sset_1,$ and
$$\widetilde \Nset=\{n_1,\ldots,n_s,\underbrace{1,\ldots,1}_\text{$r_0+r_1$};n\}.$$
Let us check first that $\widetilde\Nset$ is exact scheme. Indeed,  we have that
\begin{equation*}\label{abc}r_0+r_1=\overline n - \Nset_\#^0+(n+1) -\alpha^{0\#}_{\Nset,\Xset}=\overline {n+1}-\Nset_\#.\end{equation*}

\noindent Next, let us prove that the problem $(\widetilde\Nset,\widetilde\Xset)$ is $n$-correct. For this let us assume, in view of Proposition \ref{correct H}(ii), that
$D^{\widetilde \Nset}p|_{\widetilde\Xset}=0,$ where $p\in\Pi_n,$
and prove that $p=0.$

In view of the choice of the sets $\Sset_0$ and $\Sset_1$ we have that
$\alpha^{0\#}_{\widetilde\Nset,\widetilde\Xset}=n+1.$

\noindent Therefore, Proposition \ref{prop:line} implies that $p=\alpha_0 q,$ where
$$q\in\Pi_{n-1}\ \hbox{and}\  D^{\widehat \Nset^0}q|_{\widehat\Xset}=0.$$
Now, in view of \eqref{N0X03} and Proposition \ref{correct H}(ii), we get $q = 0,$ and hence $p = 0.$

Thus, the problem $(\widetilde\Nset,\widetilde\Xset)$ is $n$-correct, which implies that $(\Nset,\Xset)$ is $n$-solvable. \hfill \qed

From now on, in view of Step 3, we may assume that there are no three collinear points in the set $\Xset.$

{\bf Step 4.}  First note that the condition (ii) of Theorem, in view of Corollary \ref{lem:11111}, implies that $s\ge 6.$
Then consider a conic $\beta_0$ passing through the points $x^{(1)},\ldots,x^{(5)}.$ The conic $\beta_0$ is irreducible, since no three points are collinear. 

Consider the following scheme (see \eqref{delta})
$$\Nset^0:=\Nset - \Delta^{\beta^0}_\Xset=\{n_1^0,\ldots,n_s^0; n-2\}:=
\{n_1-\delta(1),\ldots, n_s-\delta(s); n-2\}.$$
We have that 
$$\sum_{i=1}^sn^0_i=\sum_{i=1}^sn_i-\sum_{i=1}^s\delta(i)\le 2n+2-5=2(n-2)+1.$$

Note that $\Nset^0$ may not satisfy the condition \eqref{dec'}.
Thus, to apply Theorem \ref{th:2n+1}, let us verify that
$$n_i^0+n_j^0:=n_i-\delta(i)+n_j-\delta(j)\le n-1\ \forall 1\le i<j\le s.$$
Indeed, if $i\le 5$ then $\delta(i)=1$ and, in view of \eqref{elln12}, this is evident. 

Now, assume conversely, that  the inequality $n_i+n_j\ge n_i^0+n_j^0\ge n$ holds for some $6\le i<j\le s.$
Then, in view of \eqref{dec'}, we have that $n_5+n_6\ge n_i+n_j\ge n,$ and therefore $n_1+\cdots+n_6\ge 3(n_5+n_6)\ge 3n.$ 

On the other hand we have that $3n\ge 2n+3,$ if $n\ge 3.$
Hence we have that $n_1+\cdots+n_6\ge 2n+3,$
 which contradicts the condition \eqref{2cond'}.

Thus, according to Theorem \ref{th:2n+1}, the problem $(\Nset^0,\Xset)$ is $(n-2)$-solvable.
Therefore, Proposition \ref{widehat} implies that there is a set $\Sset_0,$ consisting of $r_0:=\overline{n-1}- \Nset^0_\#$ simple points, such that the problem
\begin{equation}\label{n-2corr}(\widehat\Nset^0,\widehat\Xset)\ \hbox{is $(n-2)$-correct,}\end{equation} 
where  
$$\widehat \Nset^0=\{n_1^0,\ldots,n_s^0,\underbrace{1,\ldots,1}_\text{$r_0$};n-2\},\quad \widehat\Xset=\Xset\cup\Sset_0.$$  
By making use of
Remark \ref{rem:XX'} we can assume that
the set $\Sset_0$ contains no point belonging to the conic $\beta^0.$

Next, denote by $\Sset_1$ a set of $r_1:=2n + 1 -\beta^{0\#}_{\Nset,\Xset}$
simple points in $\beta^0\setminus \Xset.$ 

\noindent Consider
the problem $(\widetilde\Nset,\widetilde\Xset),$ where 
$\widetilde\Xset:=\Xset\cup \Sset_0\cup \Sset_1,$ and
$$\widetilde \Nset=\{n_1,\ldots,n_s,\underbrace{1,\ldots,1}_\text{$r_0+r_1$};n\}.$$
Let us check first that $\widetilde\Nset$ is exact scheme. Indeed,  we have that
\begin{equation*}\label{abc}r_0+r_1=\overline n - \Nset_\#^0+(n+1) -\alpha^{0\#}_{\Nset,\Xset}=\overline {n+1}-\Nset_\#.\end{equation*}
Now, let us prove that the problem $(\widetilde\Nset,\widetilde\Xset)$ is $n$-correct. For this let us assume, in view of Proposition \ref{correct H}(ii), that
$D^{\widetilde \Nset}p|_{\widetilde\Xset}=0,$ where $p\in\Pi_n$
and prove that $p=0.$

By taking into account the choice of the sets $\Sset_0$ and $\Sset_1$ we have that
$\beta^{0\#}_{\widetilde\Nset,\widetilde\Xset}=2n+1.$
Therefore, Proposition \ref{prop:conic} implies that $p=\beta^0 q,$ where
$$q\in\Pi_{n-2}\ \hbox{and}\ D^{\widehat \Nset^0}q|_{\widehat\Xset}=0.$$
Then, in view of \eqref{n-2corr} and Proposition \ref{correct H}(ii), we get $q = 0,$ and hence $p = 0.$

Thus, the problem $(\widetilde\Nset,\widetilde\Xset)$ is $n$-correct, which implies that $(\Nset,\Xset)$ is $n$-solvable. \hfill \qed

{\bf Step 5.} 
It remains to consider  the case of general dimension. 

Assume that $\Xset$ is a set of points in $k$ dimension
satisfying the $l/(n+1)$ and $c/2n+1$-conditions with respect to the problem $(\Nset,\Xset)^k.$ Then
we can project it to a set of points in  $k-1$ dimension, with the same property. Continuing in this way we will obtain a set of points $\Xset_0$ in a plane, with the same property.  Now, the problem $(\Nset,\Xset_0)^2$ is $n$-solvable, since Theorem has already been proven in the bivariate case. From here, by using $k-2$ times Remark \ref{rem:XX'}, we obtain that the original problem $(\Nset,\Xset)^k$ is $n$-solvable too.
\end{proof}

Note that Theorem \ref{th:2n+2} is not true already for the case 
\begin{equation}\label{avel}\sum_{i=1}^sn_i= 2n+3. \end{equation}
Let us bring  for this

{\bf A counterexample.}

Suppose the scheme is the following
$\Nset=\{n-1,\underbrace{1,\ldots,1}_\text{$n+4$};n\}^2,\ n\ge 3.$
Thus \eqref{avel} takes place. Also we have that $\Nset$ is $(\le)$-scheme:
$$\Nset_\#=\overline{n-1}+n+4=\overline n+4\le \overline{n+1}.$$

Suppose the corresponding set of points $\Xset$ is chosen as follows:
The points $x^{(i)},\ i=1,\ldots, n+3,$ all belong to an irredusible conic $\beta.$ The points $x^{(1)},\ x^{(n+4)},$ and  $x^{(n+5)}$ belong to a line $\alpha,$ which does not contain any other point from $\Xset.$ 

It is easily seen that the conditions (i) and (ii) of Theorem \ref{th:2n+2} are satisfied:
$$\alpha_{\Nset,\Xset}^{\#} = n + 1,\quad \beta_{\Nset,\Xset}^{\#} = 2n + 1.$$

Let us verify that the Hermite problem $(\Nset,\Xset)$
is not $n$-solvable. 

Indeed, by applying Corollary \ref{cor:line+} with respect to the line $\alpha,$ we get that the problem $(\Nset,\Xset)$ is $n$-solvable if and only if the problem $(\Nset^0,\Xset^0)$ is $(n-1)$-solvable, where $\Xset^0\subset\beta^0$ and $\Nset^0=\{n-2,\underbrace{1,\ldots,1}_\text{$n+2$};n-1\}^2.$

Now, on the  conic $\beta$ we have $n-2+n+2=2n$ points of $(\Nset^0,\Xset^0).$  Hence, in view of Corollary \ref{nes2}(ii), the problem $(\Nset^0,\Xset^0)$ is not $(n-1)$-solvable.

From Theorem \ref{th:2n+2} we get
\begin{corollary} \label{cor:2n+2} Suppose that a $(\le)$-scheme $\Nset=\{n_1,\ldots,n_s; n\}^2,$ with non-increasing sequence of multiplicities, satisfies the condition: 
$\sum_{i=1}^sn_i= 2n+2.$ 
Then $\Nset$ is regular if and only if  $s\ge 6$ and $n_1 + n_2 \le n + 1.$  
\end{corollary}
\begin{proof} Indeed, the direction ``only if" follows from Corollary \ref{cor:n1n2}. For the direction ``if" consider the problem $(\Nset,\Xset)^2,$ where in $\Xset$ no three points are collinear and no six points belong to a conic. Then according to Theorem  \ref{th:2n+2} the problem $(\Nset,\Xset)$ is $n$-solvable. 
\end{proof}
Consider the following exact scheme $\Nset=\{2,2,2,2,2;4\}^2.$ This scheme satisfies the conditions of above corollary, with one exception: $s=5.$ 
According to Corollary \ref{cor:n1n2} the scheme  $\Nset$ is not regular.

%%%%%%%%%%%%%%%%%%%%%%%%%%%%%%%%%%%%%%%%%%%%%%%%%%%%%%%%%%%%%%%%%%%%%%%%
%%%%%%%%%%%%%%%%%%%%%%%%%%%%%%%%%%%%%%%%%%%%%%%%%%%%%%%%%%%%%%%%%%%%%%%%%%
%%%%%%%%%%%%%%%%%%%%%%%%%%%%%%%%%%%%%%%%%%%%%%%%%%%%%%%%%%%%%%%%%%%%%%%%%%%%%%
%%%%%%%%%%%%%%%%%%%%%%%%%%%%%%%%%%%%%%%%%%%%%%%%%%%%%%%%%%%%%%%%%%%%%%%%%%%%%%

\end{document}